\documentclass[11pt,letterpaper]{article}
\usepackage{amsfonts, amsmath, amssymb, amscd, amsthm, graphicx, mathrsfs, wasysym, setspace, mdwlist, color, transparent, import}
\usepackage{hyperref}
\makeatletter
\def\blfootnote{\xdef\@thefnmark{}\@footnotetext}
\makeatother

\newtheorem{thm}{Theorem}[section]
\newtheorem{cor}[thm]{Corollary}
\newtheorem{lem}[thm]{Lemma}
\newtheorem{prop}[thm]{Proposition}
\theoremstyle{definition}
\newtheorem{defn}[thm]{Definition}
\newtheorem{question}[thm]{Question}
\theoremstyle{remark}
\newtheorem{rem}[thm]{Remark}
\newtheorem{ex}[thm]{Example}

\begin{document}

\title{Algebraic subgroups of acylindrically hyperbolic groups}
\author{B. Jacobson}
\date{}

\maketitle

\begin{abstract}
A subgroup of a group $G$ is called \textit{algebraic} if it can be expressed as a finite union of solution sets to systems of equations. We prove that a non-elementary subgroup $H$ of an acylindrically hyperbolic group $G$ is algebraic if and only if there exists a finite subgroup $K$ of $G$ such that $C_G(K) \leq H \leq N_G(K)$. We provide some applications of this result to free products, torsion-free relatively hyperbolic groups, and ascending chains of algebraic subgroups in acylindrically hyperbolic groups.
\end{abstract}

\vspace{-2mm} \hspace{3mm} \textbf{MSC Subject Classification:} 20F70, 20F67, 20F65.

\tableofcontents


\section{Introduction}


Let $G$ be a group and $\langle x \rangle$ be the infinite cyclic group generated by $x$. For each $g \in G$, let $\varphi_g$ denote the homomorphism $G \ast \langle x \rangle \rightarrow G$ induced by taking the identity map on $G$ and sending $x \mapsto g$. Given an element $w(x)$ of the free product $G \ast \langle x \rangle$, there is a corresponding function $G \rightarrow G$ whose evaluation at $g \in G$ is $$w(g) : = \varphi_g(w(x)).$$ Building on this notion, we write $w(x)=1$ to represent a \textit{group equation} in the single variable $x$ with coefficients in $G$ whose \textit{solution set} is $$\{ g \in G \, | \, w(g)=1 \}.$$ The solution set above is called the  \textit{primitive solution set corresponding to $w(x)$}.

\begin{defn}
The \textit{Zariski topology} (or \textit{verbal topology} as in \cite{Bryant}) on $G$ is defined by taking the collection of primitive solution sets to be a sub-basis for the closed sets of the topology. That is, each Zariski-closed set of $G$ is of the form $$ \cap_{i \in I} S_i $$ for some indexing set $I$, where for each $i \in I$, the set $S_i$ is a finite union of primitive solution sets corresponding to (single-variable) group equations with coefficients in $G$. Note that in general, the Zariski topology is not a group topology.
\end{defn}

The Zariski topology (and its more general form in \cite{AG1}) is useful in bringing notions from algebraic geometry to the group theoretical setting in order to aid in the study of equations over groups. The topology is also useful in the study of topological groups: Zariski-closed sets are closed in every $T_0$ group topology, and, in the case of countable groups, the Zariski-closed sets are the only such sets \cite{Markov}. For some recent applications to topological groups, see \cite{top} and references therein.

The following definition is motivated by the standard notion of an algebraic subgroup in algebraic geometry.

\begin{defn}
A Zariski-closed subgroup (or more generally, a subset) of $G$ is called \textit{algebraic}.
\end{defn}

\begin{ex}\label{centalg}
For any subgroup $H \leq G$, the centralizer $C_G(H)$ is algebraic, as $$C_G(H)= \cap_{h \in H} \{ x \in G \, | \, [x,h]=1 \}.$$
\end{ex}

In \cite{KM}, Kharlampovich and Myasnikov show that given a torsion-free non-elementary  hyperbolic group $G$, any proper, first-order definable subgroup $H  \leq G$ is cyclic. In the language of \cite{AG1}, $G$ is a domain, so any Zariski-closed subset of $G$ may be expressed as an intersection of primitive solution sets. Furthermore, since $G$ is equationally Noetherian (\cite[Thm. 1.22]{S}), this intersection may be taken to be finite, so algebraic subgroups of $G$ are first-order definable, and in particular, the structural result from \cite{KM} for definable subgroups of $G$ holds for algebraic subgroups. In this paper, we obtain a structural result for algebraic subgroups in the more general case when $G$ is an acylindrically hyperbolic group. Before stating this result, we briefly discuss acylindrically hyperbolic groups.

Let $G$ be group acting on a metric space $(S,d)$. Note that in this paper, all actions of groups on metric spaces are assumed to be isometric.

\begin{defn} \label{acyl}
The action of $G$ on $S$ is called \textit{acylindrical} if for each $\varepsilon > 0$ there exist $R, N >0$ such that for every pair of points $x,y \in S$ with $d(x,y) \geq R$, there are at most $N$ elements $g \in G$ such that \[ d(x,gx) \leq \varepsilon \quad \textrm{and} \quad d(y,gy) \leq \varepsilon \] For example, every geometric (i.e. proper and cobounded) action is acylindrical. However, acylindricity is a much weaker condition. E.g., every action on a bounded space is acylindrical. 
\end{defn}

\begin{defn}\label{defn2}
If $S$ is a hyperbolic space, then the action of $G$ on $S$ is called \textit{elementary} if the limit set of $G$ on the Gromov boundary $\partial S$ contains at most $2$ points.
\end{defn}

A group $G$ is called \textit{acylindrically hyperbolic} if it admits a non-elementary acylindrical action on a hyperbolic space. (More about groups acting acylindrically on hyperbolic spaces, including other equivalent definitions of acylindrically hyperbolic groups, can be found in Subsection \ref{acting acyl}. For a more comprehensive discussion of acylindrically hyperbolic groups, see \cite{Osin}.) The class of acylindrically hyperbolic groups contains many interesting examples, including non-elementary hyperbolic and relatively hyperbolic groups, infinite mapping class groups of punctured closed surfaces, $Out(F_n)$ for $n \geq 2$, directly indecomposable 
non-cyclic right angled Artin groups, most $3$-manifold groups (see \cite[Appendix]{Osin} for details), and groups of deficiency $\geq2$ \cite{OsinL2}.

\begin{defn}\label{non-elem}
Given an acylindrically hyperbolic group $G$, a subgroup $H \leq G$ is called \textit{non-elementary} if for some acylindrical action of $G$ on a hyperbolic space $S$, the action of $H$ on $S$ is non-elementary.
\end{defn}

\begin{ex}
Let $G = H \ast \mathbb{Z}$, where $\mathbb{Z}=\langle z \rangle$ and $H=F(x,y)$ is a free group of rank $2$. Then the action of $G$ on the Bass-Serre tree associated to the free product structure is acylindrical and non-elementary, but the induced action of $H$ is elementary. However, $H$ is a non-elementary subgroup of $G$ since there exists another acylindrical action of $G$ on a hyperbolic space (namely the action of $G$ on its Cayley graph with respect to the generating set $\{x,y,z\}$) such that $H$ acts non-elementarily.
\end{ex}

We are now ready to state the main result.

\begin{thm}\label{main}
Suppose that $G$ is an acylindrically hyperbolic group and that $H \leq G$ is non-elementary. Then $H$ is algebraic if and only if there exists a finite subgroup $K$ of $G$ such that $C_G(K) \leq H \leq N_G(K)$.
\end{thm}

It should be noted that we actually prove a stronger version of the forward implication of Theorem \ref{main} (see Theorem \ref{strongmain}) which is a result about the Zariski closure of non-elementary subgroups of acylindrically hyperbolic groups. In the course of proving Theorem \ref{main}, we obtain a technical result (Proposition \ref{coboundedness}) which seems to be of independent interest.

If $H$ is an algebraic subgroup of a torsion-free non-elementary hyperbolic group $G$, then either the action of $H$ on the Cayley graph of $G$ is elementary (in which case $H$ is either $\{1\}$ or infinite cyclic) or $K=\{1\}$ so that $H=G$. Thus Theorem \ref{main} recovers the structural result for algebraic subgroups implied by the result of Kharlampovich and Myasnikov in \cite{KM}.

Theorem \ref{main} also yields the following corollaries, which we prove in Section \ref{appl}:

\begin{cor}\label{freecor}
Let $A$ and $B$ be non-trivial groups, and let $H$ be an algebraic subgroup of $A \ast B$. Then at least one of the following holds:
\begin{itemize}
\item[(a)] $H$ is either infinite cyclic or isomorphic to $D_\infty$, the infinite dihedral group.
\item[(b)] $H$ is conjugate to a subgroup of either $A$ or $B$.
\item[(c)] $H=A \ast B$.
\end{itemize}
\end{cor}

In the case of torsion-free relatively hyperbolic groups, we obtain the following.

\begin{cor}\label{TFRHcor}
Let $G$ be a torsion-free relatively hyperbolic group with peripheral subgroups $\{H_\lambda\}_{\lambda \in \Lambda}$, and let $H \leq G$ be an algebraic subgroup. Then at least one of the following holds:
\begin{itemize}
\item[(a)] $H=G$.
\item[(b)] $H$ is cyclic.
\item[(c)] $H$ is conjugate to a subgroup of some $H_\lambda$.
\end{itemize}
Furthermore, if (c) holds for an abelian $H_\lambda$, then either $H=\{1\}$ or $H$ is conjugate to $H_\lambda$.
\end{cor}

Important examples of torsion-free relatively hyperbolic groups with abelian peripheral subgroups (the so-called toral relatively hyperbolic groups) include limit groups in the sense of Sela and fundamental groups of hyperbolic knot compliments. Limit groups are hyperbolic relative to their maximal abelian non-cyclic subgroups (see \cite[Thm. 0.3]{Dahmani}), while fundamental groups of hyperbolic knot compliments are hyperbolic relative to free abelian subgroups of rank $2$ (cusp subgroups; see \cite{Farb}).

\begin{cor}\label{ACCcor}
Let $G$ be an acylindrically hyperbolic group and let $H_1 \leq H_2 \leq H_3 \leq \ldots$ be an ascending chain of algebraic subgroups of $G$. Then either
\begin{itemize}
\item[(a)] for each acylindrical action of $G$ on a hyperbolic space $S$, the subgroup $\cup_{i \in \mathbb{N}} H_i$ acts on $S$ with bounded orbits (in particular, each $H_i$ is elliptic), or
\item[(b)] the chain stabilizes.
\end{itemize}
\end{cor}

Note that in general, it is actually possible for an acylindrically hyperbolic group to have an ascending chain of (elliptic) algebraic subgroups which does not stabilize. Example \ref{ACCex} provides a construction of such a group.

{\bf Acknowledgments.} I would like to thank Denis Osin for his guidance throughout this project, Sahana Balasubramanya for her careful reading and comments on a draft of this paper, and Michael Hull for a helpful discussion which led to the removal of some non-essential material. I would also like to thank the anonymous referee for a number of useful comments and suggestions.


\section{Preliminaries}


\subsection{Algebraic subgroups}

Let $G$ be a group, and let $\langle x \rangle$ be the infinite cyclic group generated by $x$. Notice that for two distinct elements $u(x), v(x) \in G \ast \langle x \rangle$,  the equations $u(x)=1$ and $v(x)=1$ may have the same solution set. (For example, let $v(x)=w(x)^{-1}u(x)w(x)$ for any $w(x) \in G \ast \langle x \rangle$.) In this case we will regard the equations $u(x)=1$ and $v(x)=1$ as equivalent, as they both define the same subbasic closed set of the Zariski topology.

It is helpful, given a (not necessarily finite) generating set $X$ of $G$, to view an element $w(x) \in G \ast \langle x \rangle$ as a word in the elements of $X \cup \{x\}$ and their inverses. In particular, note that for any cyclic reduction $w'(x)$ of the word $w(x)$, $w(x)=1$ and $w'(x)=1$ are the same equation and that therefore, we may always assume that the word on the left-hand side of an equation is cyclically reduced.

\begin{rem} \label{multhomeom}
Viewing each $w(x) \in G \ast \langle x \rangle$ as a word in $(X \cup \{x \})^{\pm1}$ allows us to observe that for any $g \in G$, the map $G \rightarrow G$ given by left (similarly, right) multiplication by $g$ is a homeomorphism with respect to the Zariski topology. In particular, for each $g \in G$ and each word $w(x) \in G \ast \langle x \rangle$, let $w(g^{-1}x)$ be the word obtained from $w(x)$ by replacing each instance of the letter $x$ with $g^{-1}x$ and each instance of the letter $x^{-1}$ with $(g^{-1}x)^{-1}=x^{-1}g$, and note that if $S$ is the solution set of the equation $w(x)=1$, then $gS$ is the solution set of the equation $w(g^{-1}x)=1$.
\end{rem}

\begin{cor}\label{finindalg}
If a subgroup $H \leq G$ contains a finite index algebraic subgroup $A$, then $H$ is also algebraic.
\end{cor}

\begin{proof}
$H$ is a finite union of cosets $h_1 A, \ldots, h_n A$ ($h_1, \ldots, h_n \in H$), each of which is algebraic by Remark \ref{multhomeom}.
\end{proof}

\subsection{Hyperbolically embedded subgroups} 

Let $G$ be a group, and let $\{ H_{\lambda} \}_{\lambda \in \Lambda}$ be a collection of subgroups of $G$. Given a subset $X \subseteq G$ such that $G$ is generated  by $X \cup (\cup_{\lambda \in \Lambda} H_{\lambda})$, let $\Gamma (G, X \sqcup \mathcal{H})$ denote the Cayley graph of $G$ whose edges are labeled by letters from the alphabet $X \sqcup \mathcal{H}$, where \[ \mathcal{H} = \bigsqcup_{\lambda \in \Lambda} H_{\lambda}. \] It is important to note the use of disjoint unions here. In particular, if $g \in G$ appears multiple times in $X \sqcup \mathcal{H}$, then we include edges in $\Gamma (G, X \sqcup \mathcal{H})$ corresponding to \textit{each} of the distinct copies of $g$. We think of the Cayley graphs $\Gamma (H_{\lambda}, H_{\lambda})$ as complete subgraphs of $\Gamma (G, X \sqcup \mathcal{H})$.

\begin{defn}\label{dhat}
(\cite[Defn. 4.2]{DGO}) For each $\lambda \in \Lambda$, define the relative metric $\widehat{d}_{\lambda}: H_{\lambda} \times H_{\lambda} \rightarrow [0, \infty]$ as follows. A (combinatorial) path $p$ in $\Gamma (G, X \sqcup \mathcal{H})$ is called \textit{$\lambda$-admissible} if it contains no edges of $\Gamma (H_{\lambda}, H_{\lambda})$. Importantly, $p$ may still contain vertices in $\Gamma (H_{\lambda}, H_{\lambda})$ as well as edges labeled by letters from $H_{\lambda}$ (provided those edges are not in $\Gamma (H_{\lambda}, H_{\lambda})$). For each $h,k \in H_{\lambda}$, define $\widehat{d}_{\lambda} (h,k)$ to be the length of a shortest $\lambda$-admissible path in $\Gamma (G, X \sqcup \mathcal{H})$ connecting $h$ to $k$ if such a path exists. Otherwise, let $\widehat{d}_{\lambda} (h,k) = \infty$. It is easy to see that $\widehat{d}_{\lambda}$ satisfies the triangle inequality. The relative metric $\widehat{d}_{\lambda}$ is extended to $G$ by setting \begin{equation}\nonumber \widehat{d}_{\lambda}(f,g) := \left\{ \begin{array}{ll}
       \widehat{d}_{\lambda}(f^{-1}g,1), & \textrm{if} \, f^{-1}g \in H_{\lambda};\\
       \infty, &  \textrm{otherwise.}
     \end{array}
   \right. \end{equation}
\end{defn}

\begin{rem}
The extension of  $\widehat{d}_{\lambda}$ to $G$ is a matter of notational convenience for the statement and use of Lemma \ref{ngonC}.
\end{rem}

\begin{defn}   	
(\cite[Defn. 4.25]{DGO}) Let $G$ be a group and $X \subseteq G$. A collection of subgroups $\{ H_{\lambda} \}_{\lambda \in \Lambda}$ of $G$ is said to be \textit{hyperbolically embedded in $G$ with respect to $X$} (notated $\{ H_{\lambda} \}_{\lambda \in \Lambda} \hookrightarrow_h (G,X)$) if the following conditions hold:
\begin{itemize}
\item[(a)] $G$ is generated  by $X \cup (\cup_{\lambda \in \Lambda} H_{\lambda})$
\item[(b)] The Cayley graph $\Gamma (G, X \sqcup \mathcal{H})$ is hyperbolic.
\item[(c)] For each $\lambda \in \Lambda$ the metric space $(H_{\lambda}, \widehat{d}_{\lambda})$ is proper, i.e., any ball of finite radius contains only finitely many elements.
\end{itemize}
Furthermore, $\{ H_{\lambda} \}_{\lambda \in \Lambda}$ is said to be \textit{hyperbolically embedded} in $G$ (notated $\{ H_{\lambda} \}_{\lambda \in \Lambda} \hookrightarrow_h G$) if there exists $X \subseteq G$ such that $\{ H_{\lambda} \}_{\lambda \in \Lambda} \hookrightarrow_h (G,X)$.
\end{defn}

\begin{lem}\label{finitehypemb}
(\cite[Cor. 4.27]{DGO}) Let $G$ be a group $\{ H_{\lambda} \}_{\lambda \in \Lambda}$ a collection of subgroups of $G$, $X_1, X_2 \subseteq G$ such that $$G = \langle X_1 \cup (\cup_{\lambda \in \Lambda} H_{\lambda}) \rangle =\langle X_2 \cup (\cup_{\lambda \in \Lambda} H_{\lambda}) \rangle$$ and $|X_1 \bigtriangleup X_2| < \infty$. Then $\{ H_{\lambda} \}_{\lambda \in \Lambda} \hookrightarrow_h (G,X_1)$ if and only if $\{ H_{\lambda} \}_{\lambda \in \Lambda} \hookrightarrow_h (G,X_2)$.
\end{lem}

Observe in particular that if $\{ H_{\lambda} \}_{\lambda \in \Lambda}$ is hyperbolically embedded in $G$ with respect to some $X$, then $\{ H_{\lambda} \}_{\lambda \in \Lambda}$ is also hyperbolically embedded in $G$ with respect to any set $X'$ obtained from $X$ by adding finitely many elements of $G$.
   
\begin{defn}
(\cite[Defns. 2.19, 2.20]{ORel}) Let $\{ H_{\lambda} \}_{\lambda \in \Lambda} \hookrightarrow_h (G,X)$ and let $q$ be a path in $\Gamma (G, X \sqcup \mathcal{H})$. A non-trivial subpath $p$ of $q$ is called an \textit{$H_{\lambda}$-subpath} if the label of $p$ is a word in the alphabet $H_{\lambda}$. An $H_{\lambda}$-subpath $p$ of $q$ is called an \textit{$H_{\lambda}$-component} if it is not contained in a longer $H_{\lambda}$-subpath of $q$, and, in the case that $q$ is  loop, $p$ is not contained in a longer $H_{\lambda}$-subpath of any cyclic shift of $q$.

Two $H_{\lambda}$-components $p_1,p_2$ of $q$ are called \textit{connected} if there exists a path $c$ in $\Gamma (G, X \sqcup \mathcal{H})$ that connects some vertex of $p_1$ to some vertext of $p_2$ and the label of $c$ is a word consisting only of letters from $H_{\lambda}$. Algebraically, this means that all vertices of $p_1$ and $p_2$ belong to the same left coset of $H_{\lambda}$. An $H_{\lambda}$-component $p$ of $q$ is called \textit{isolated} in $q$ if it is not connected to any other $H_{\lambda}$-component of $q$.
\end{defn}

For a path $p$ in $\Gamma (G, X \sqcup \mathcal{H})$, let $p_{-}$ and $p_{+}$ denote the initial and terminal vertices of $p$ respectively. The following lemma is a simplification of \cite[Prop. 4.14]{DGO} and is integral to the technique of the main proof:

\begin{lem}\label{ngonC}
Let $\{ H_{\lambda} \}_{\lambda \in \Lambda} \hookrightarrow_h (G,X)$. Then there exists a constant $C>0$ such that for any $n$-gon
$p$ with geodesic sides in $\Gamma (G, X \sqcup \mathcal{H})$, any $\lambda \in \Lambda$, and any isolated $H_{\lambda}$-component $a$ of $p$, we have $\widehat{d}_{\lambda}(a_{-},a_{+}) \leq Cn$.
\end{lem}

\subsection{Groups acting on hyperbolic spaces} \label{acting acyl}

\begin{defn}\label{defn1}
Given a group $G$ acting on a hyperbolic space $S$, an element $g \in G$ is called \textit{loxodromic} if the map $\mathbb{Z} \rightarrow S$ given by $n \mapsto g^n s$ is a quasi-isometric embedding for some (equivalently, any) $s \in S$. Each loxodromic element $g \in G$ has exactly two limit points $g^{\pm \infty}$ on the Gromov boundary $\partial S$. A pair of loxodromic elements $g,h \in G$ are called \textit{independent} if the sets $\{g^{\pm \infty} \}$ and $\{h^{\pm \infty} \}$ are disjoint.
\end{defn}

The following definition is due to Bestvina and Fujiwara in \cite{BF}:

\begin{defn}
For a group $G$ acting on a metric space $S$, an element $h \in G$ satisfies the \textit{weak proper discontinuity} condition (or $h$ is a \textit{WPD element}) if for each $\varepsilon > 0$ and each $x \in S$ there exists $M \in \mathbb{N}$ such that \[ | \{ g \in G \, | \, d(x,gx) < \varepsilon, d(h^M x, gh^M x) < \varepsilon \} | < \infty \] Let $\mathcal{L}_{WPD}(G,S)$ denote the set of elements $g \in G$ which are loxodromic WPD with respect to the action of $G$ on $S$. It is easy to see that if $G$ acts acylindrically (as in Definition \ref{acyl}) and $S$ is a hyperbolic space, then every loxodromic element is a WPD element.
\end{defn}

\begin{lem}\label{L1} (\cite[Lem. 6.5, Cor. 6.6]{DGO}) Let $G$ be a group acting on a hyperbolic space $S$. Each element $h \in \mathcal{L}_{WPD}(G,S)$ is contained in a unique maximal virtually cyclic subgroup of $G$ which is denoted by $E_G(h)$. Moreover, for every $g \in G$ the following are equivalent:
\begin{itemize}
\item[(i)] $g \in E_G(h)$.
\item[(ii)] $gh^n g^{-1}= h^{\pm n}$ for some $n \in \mathbb{N}$.
\item[(iii)] $gh^k g^{-1}= h^l$ for some $k,l \in \mathbb{Z} \backslash \{ 0 \}$.
\end{itemize}
\end{lem}

\begin{cor}
For each $h \in \mathcal{L}_{WPD}(G,S)$, the subgroup $E_G(h)$ is algebraic.
\end{cor}

\begin{proof}
It is not hard to see that $\langle h \rangle$ is a finite index subgroup of $E_G(h)$ so that $E_G(h)$ is a finite union of cosets $k_1 \langle h \rangle, \ldots, k_n \langle h \rangle$ for some $k_1, \ldots, k_n \in E_G(h)$. Noting additionally that $C_G(h) \subseteq E_G(h)$ by (ii) of Lemma \ref{L1}, we have $$E_G(h)=\bigcup_{i=1} ^n k_i \cdot C_G(h)= \bigcup_{i=1} ^n k_i \cdot \{ x \in G \, | \, [x,h]=1 \} = \bigcup_{i=1} ^n \{ x \in G \, | \, [k_i ^{-1} x,h]=1 \}.$$
\end{proof}

\begin{thm}\label{actions} (\cite[Thm. 1.1]{Osin})
Let $G$ be a group acting acylindrically on a hyperbolic space. Then $G$ satisfies exactly one of the following three conditions:
\begin{enumerate}
\item[(a)] G has bounded orbits.
\item[(b)] G is virtually cyclic and contains a loxodromic element.
\item[(c)] G contains infinitely many independent loxodromic elements.
\end{enumerate}
\end{thm}

\noindent If $G$ has bounded orbits, $G$ is called \textit{elliptic}.

\begin{rem}\label{loxrem}
In light of Definitions \ref{defn2} and \ref{defn1}, Theorem \ref{actions} tells us that if a group $G$ acts acylindrically on a hyperbolic space $S$, then the action is non-elementary if and only if $G$ is not virtually cyclic and contains at least one loxodromic element.
\end{rem}

\begin{thm}\label{AH} (\cite[Thm. 1.2]{Osin})
For any group $G$, the following conditions are equivalent:
\begin{enumerate}
\item[(i)] $G$ admits a non-elementary acylindrical action on a hyperbolic space.
\item[(ii)] There exists a generating set $X$ of $G$ such that the corresponding Cayley graph $\Gamma (G,X)$ is hyperbolic, $|\partial \Gamma (G,X) |>2$, and the natural action of $G$ on $\Gamma (G,X)$ is acylindrical.
\item[(iii)] $G$ is not virtually cyclic and admits an action on a hyperbolic space such that at least one element of $G$ is loxodromic and satisfies the WPD condition.
\item[(iv)] $G$ contains a proper, infinite, hyperbolically embedded subgroup.
\end{enumerate}
\end{thm}

\begin{defn}
A group $G$ is called \textit{acylindrically hyperbolic} if it satisfies one of the four equivalent conditions above.
\end{defn}


\section{Proof of the main result}


\subsection{Constructing a cobounded action}\label{keysub}

The proof of our main result relies on the existence of arbitrarily large finite collections of loxodromic elements in the non-elementary subgroup $H$ with certain useful properties. Lemma \ref{hlem} guarantees these elements, provided we can find a generating set of $G$ satisfying the requisite assumptions. To find such a generating set, we prove Proposition \ref{coboundedness}, which also seems to be of independent interest.

\begin{prop}\label{coboundedness}
Suppose that a group $G$ acts on a hyperbolic space $S$ and that $H$ is a subgroup of $G$ such that $H$ is not virtually cyclic and $H \cap \mathcal{L}_{WPD}(G,S) \neq \emptyset$. Then there exists a generating set $X$ of $G$ such that the Cayley graph $\Gamma(G, X)$ is hyperbolic, the action of $G$ on $\Gamma(G, X)$ is acylindrical, and the action of $H$ on $\Gamma(G, X)$ is non-elementary.
\end{prop}

To prove Proposition \ref{coboundedness}, we need three other results.

\begin{lem}\label{loxhyp}
(\cite[Cor. 2.9]{DGO}) Let $G$ be a group acting on a hyperbolic space. For each loxodromic WPD element $h \in G$, we have $E_G(h) \hookrightarrow_h G$.
\end{lem}

\begin{lem}\label{5.4}
(\cite[Thm. 5.4]{Osin}) Let $G$ be a group, $\{H_{\lambda} \}_{\lambda \in \Lambda}$ a finite collection of subgroups of $G$, $X \subseteq G$, $\mathcal{H}= \bigsqcup_{\lambda \in \Lambda} H_\lambda$. Suppose that $\{H_{\lambda} \}_{\lambda \in \Lambda} \hookrightarrow_h (G,X)$. Then there exists $Y \subseteq G$ such that $X \subseteq Y$ and the following conditions hold:
\begin{itemize}
\item[(a)] $\{H_{\lambda} \}_{\lambda \in \Lambda} \hookrightarrow_h (G,Y)$. In particular, the Cayley graph $\Gamma(G,Y \sqcup \mathcal{H})$ is hyperbolic.
\item[(b)] The action of $G$ on $\Gamma(G,Y \sqcup \mathcal{H})$ is acylindrical.
\end{itemize}
\end{lem}

\begin{lem}\label{6.12}
(\cite[Cor. 6.12]{DGO}) Let $G$ be a group, $X \subseteq G$, $\ H \hookrightarrow_h (G,X)$ a proper, infinite subgroup. Then for every $a \in G \backslash H$, there exists $k \in H$ such that $ak$ is loxodromic and satisfies WPD with respect to the action of $G$ on $\Gamma(G,X \sqcup H)$.

If, in addition, $H$ is finitely generated and contains an element $h$ of infinite order, then we can choose $k$ to be a power of $h$.
\end{lem}

We are now ready to prove Proposition \ref{coboundedness}.

\begin{proof}[Proof of Proposition \ref{coboundedness}]
Let $h \in H \cap \mathcal{L}_{WPD}(G,S)$. By Lemma \ref{loxhyp} there exists $Z \subseteq G$ such that $E_G(h) \hookrightarrow_h (G,Z)$. By Lemma \ref{5.4} there exists $Y \subseteq G$ such that $Z \subseteq Y$ and
\begin{itemize}
\item[(a)] $E_G(h) \hookrightarrow_h (G,Y)$. In particular, the Cayley graph $\Gamma(G, Y\sqcup E_G(h))$ is hyperbolic.
\item[(b)] The action of $G$ on $\Gamma(G, Y\sqcup E_G(h))$ is acylindrical.
\end{itemize}
Now $E_G(h)$ is clearly infinite (since $h$ is of infinite order), and $H \backslash E_G(h) \neq \emptyset$ (since $H$ is not virtually cyclic). Furthermore, since $E_G(h)$ is virtually cyclic by definition, it is finitely generated. Therefore, we may apply Lemma \ref{6.12} to $E_G(h) \hookrightarrow_h (G,Y)$ and any element \linebreak $a \in H \backslash E_G(h)$ to conclude that there exists $n \in \mathbb{Z}$ such that $ah^n$ is loxodromic with respect to the action of $G$ on $\Gamma(G, Y\sqcup E_G(h))$. Then since $a, h \in H$, we have that $ah^n \in H$, so $H$ contains an element which is loxodromic with respect to the action of $G$ on $\Gamma(G, Y\sqcup E_G(h))$. By Remark \ref{loxrem}, the fact that $H$ is not virtually cyclic and contains a loxodromic element with respect to the action of $G$ on $\Gamma(G, Y\sqcup E_G(h))$ implies that the action of $H$ on $\Gamma(G, Y\sqcup E_G(h))$ is non-elementary.

Set $X = Y \cup E_G(h)$ and observe that the hyperbolicity of the Cayley graph $\Gamma(G,X)$, the acylindricity of the the action of $G$ on $\Gamma(G,X)$, and the fact that the action of $H$ on $\Gamma(G,X)$ is non-elementary all follow from the $G$-equivariant quasi-isometry $\Gamma(G, Y\sqcup E_G(h)) \rightarrow \Gamma(G,X)$ given by fixing vertices and unifying duplicate copies of edges wherever necessary.
\end{proof}

The following is the first part of \cite[Lem. 5.5]{Hull} which, together with Proposition \ref{coboundedness}, yields a useful lemma:

\begin{lem}\label{L5}
Let Let $G$ be a group and $X \subseteq G$ be a generating set of $G$ such that the Cayley graph $\Gamma(G,X)$ is hyperbolic and the action of $G$ on $\Gamma(G,X)$ is acylindrical. If $H \leq G$ acts non-elementarily on $\Gamma(G,X)$, then there exists a unique maximal finite subgroup of $G$ which is normalized by $H$, denoted $K_G(H)$.
\end{lem}

Combining this result with Proposition \ref{coboundedness}, we can obtain $K_G(H)$ for any non-elementary subgroup $H$.

\begin{lem}\label{uniquenormalized}
Suppose that $G$ is an acylindrically hyperbolic group and $H$ is a non-elementary subgroup of $G$. Then there exists a unique maximal finite subgroup of $G$ which is normalized by $H$, denoted $K_G(H)$.
\end{lem}

\begin{proof}
Since $G$ acts acylindrically on a hyperbolic space $S$ such that the action of $H$ on $S$ is non-elementary, we have by Remark \ref{loxrem} that $H$ is not virtually cyclic and $H \cap \mathcal{L}_{WPD}(G,S) \neq \emptyset$. So by Proposition \ref{coboundedness}, there exists generating set $X$ of $G$ such that the Cayley graph $\Gamma(G, X)$ is hyperbolic, the action of $G$ on $\Gamma(G, X)$ is acylindrical, and the action of $H$ on $\Gamma(G, X)$ is non-elementary. Then, by Lemma \ref{L5}, there exists a unique maximal finite subgroup $K_G(H)$ of $G$ which is normalized by $H$.
\end{proof}


\subsection{Proof of Theorem \ref{main}} \label{big proof}


We begin by proving a technical lemma:

\begin{lem}\label{tech}
Let $G$ be a group, $\{ H_{\lambda} \}_{\lambda \in \Lambda}$ a collection of subgroups of $G$, $\mathcal{H} = \bigsqcup_{\lambda \in \Lambda} H_\lambda$, and $X \subseteq G$. Suppose that $\{ H_{\lambda} \}_{\lambda \in \Lambda} \hookrightarrow_h (G,X)$, $n \geq 2$, $g_1, g_2, \ldots, g_n \in X$, and $h_1, h_2, \ldots, h_n \in H_\lambda$ for some $\lambda \in \Lambda$ such that $g_1 h_1 g_2 h_2 \ldots g_n h_n =1$. If for each $i$ we have $\widehat{d}_{\lambda}(h_i,1)>2Cn$, where $\widehat{d}_{\lambda}$ is defined as in Definition \ref{dhat} and $C$ is the constant given by Lemma \ref{ngonC}, then there exists $j \in \{1, \ldots, n \}$ such that $g_j \in H_\lambda$.
\end{lem}

\begin{proof}
Since $g_1 h_1 g_2 h_2 \ldots g_n h_n =1$, there exists a geodesic $2n$-gon $$p=a_1 b_1 a_2 b_2 \ldots a_n b_n$$ in the Cayley graph $\Gamma (G, X \sqcup \mathcal{H})$, where each side $a_i$ is an edge labeled by $g_i \in X$ and each side $b_i$ is an edge labeled by $h_i \in H_\lambda$ so that the closed loop $p$ corresponds to the relation $g_1 h_1 g_2 h_2 \ldots g_n h_n=1$. Then for each $i$, the side $b_i$ is an $H_{\lambda}$-component of $p$.

It suffices to show that there exists $j \in \{1, \ldots, n \}$ such that the two $H_{\lambda}$-components adjacent to $a_j$ are connected. If this is the case, then there exists an edge of $\Gamma (G, X \sqcup \mathcal{H})$ with label in $H_\lambda$ which connects $a_{j-}$ and $a_{j+}$ so that $g_j \in H_\lambda$,  finishing the proof.

Now for each $i$, we have $\widehat{d}_{\lambda}(h_i,1)>2Cn$, so Lemma \ref{ngonC} implies that none of the components $b_1, \ldots, b_n$ are isolated in $p$. Let $b_r$ and $b_s$ be connected $H_{\lambda}$-components such that the number, $N$, of sides of $p$ between $b_{r+}$ and $b_{s-}$ is minimal. If $N=1$, then choose $a_j$ to be $a_s$, where $b_{r+}=a_{s-}$ and $b_{s-}=a_{s+}$.

Otherwise, there exists $b_t$ which is an edge on the subpath of $p$ (or cyclic shift of $p$) between $b_{r+}$ and $b_{s-}$. Since $b_r$ and $b_s$ are connected, there exists an edge $e$ of $\Gamma (G, X \sqcup \mathcal{H})$ with label in $H_\lambda$ which connects $b_{r+}$ and $b_{s-}$.  Let $p'$ be the geodesic $(N+1)$-gon formed by $e$ and the sides of $p$ between $b_{r+}$ and $b_{s-}$. Then $N+1<2n$, so $$\widehat{d}_{\lambda}(h_t,1)> 2Cn > C(N+1) $$ so that by Lemma \ref{ngonC}, $b_t$ is not an isolated $H_{\lambda}$-component of $p'$.

\begin{figure}
\centering
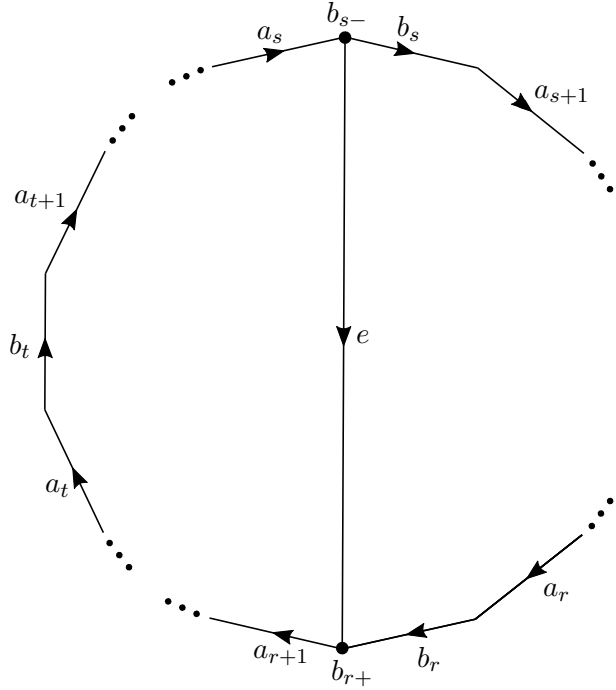
\caption{The $2n$-gon $p$ with additional edge $e$}
\end{figure}

If $b_t$ is connected to any other $H_{\lambda}$-component of $p'$, then clearly $N$ was not minimal, and if $b_t$ is connected to $e$, then it is connected to both $b_r$ and $b_s$, so again $N$ was not minimal. Hence it must be that $N=1$, finishing the proof.

\end{proof}

We will also need two more results for the main proof. Recall that two elements $g,h$ of a group $G$ are called \textit{commensurable} if there exist $x \in G$ and $k,l \in \mathbb{Z} \backslash \{0 \}$ such that $g^k = xh^l x^{-1}$ and called \textit{non-commensurable} otherwise.

\begin{lem}\label{T3}
(\cite[Thm. 6.8]{DGO}) Let $G$ be a group acting on a hyperbolic space $S$. If $h_1, \ldots, h_k$ are pairwise non-commensurable loxodromic WPD elements with respect to the action of $G$ on $S$ then there exists a set $X \subseteq G$ such that $\{E_G(h_1), \ldots, E_G(h_k)\} \hookrightarrow_h (G,X)$.
\end{lem}

\begin{lem}\label{hlem}
(\cite[Lem. 5.6]{Hull}) Let Let $G$ be a group and $X \subseteq G$ be a generating set of $G$ such that the Cayley graph $\Gamma(G,X)$ is hyperbolic and the action of $G$ on $\Gamma(G,X)$ is acylindrical. If $H \leq G$ acts non-elementarily on $\Gamma(G,X)$, then for each $k \in \mathbb{N}$, there exist pairwise non-commensurable loxodromic elements $a_1, \ldots, a_k \in H$ such that $E_G(a_i)=\langle a_i \rangle \times K_G(H)$.
\end{lem}

It will be helpful in the next proof (and later in the proof of Lemma \ref{abelhypemb}) to have the following definition:

\begin{defn}\label{equiv}
Let $G$ be a group and $S \subseteq G$. We say that two equations $p(x)=1$ and $q(x)=1$ are \textit{$S$-equivalent} if for each $s \in S$ we have that $p(s)=1$ if and only if $q(s)=1$.
\end{defn}

We can now prove the following result:

\begin{thm}\label{strongmain}
Suppose that $G$ is an  acylindrically hyperbolic group and $H$ is a non-elementary subgroup of $G$. Then the Zariski closure of $H$ contains $C_G(K_G(H))$.
\end{thm}

\begin{proof}
Let $V$ be the Zariski closure of $H$. Then $$V = \cap_{\lambda \in \Lambda} V_\lambda$$ where $\Lambda$ is some indexing set and each $V_\lambda$ is a union of finitely many primitive solution sets. We want to show that $C_G(K_G(H)) \subseteq V_\lambda$ for each $\lambda \in \Lambda$. So let $p_1(x), \ldots, p_l(x)$ be elements of $G \ast \langle x \rangle$ such that $$H \subseteq V_\lambda = \bigcup_{i = 1} ^l \{ g \in G \, | \, p_i(g)=1 \}.$$
Without loss of generality, we may assume that for each $i$, $p_i(x)$ is cyclically reduced. Furthermore, we may assume (by taking a cyclic permutation of $p(x)$ if necessary) that $p_i(x)$ is of the form $$p_i(x)=g_{1,i} x^{\alpha_{1,i}} g_{2,i} x^{\alpha_{2,i}} \ldots g_{m_i,i} x^{\alpha_{m_i,i}}$$ where $g_{1,i}, \ldots, g_{m_i,i} \in G$ and $\alpha_{1,i}, \ldots, \alpha_{m_i,i} \in \mathbb{Z}$, and if $m_i \geq 2$, then $g_{1,i}, \ldots, g_{m_i,i} \neq 1$ and $\alpha_{1,i}, \ldots, \alpha_{m_i,i} \neq 0$.

Set
\begin{equation}\label{keqn}
k=l  \cdot (\max \{m_1, \ldots, m_l \}+1).
\end{equation}
By assumption, $G$ acts acylindrically on a hyperbolic space $S$ such that the action of $H$ on $S$ is non-elementary. By Remark \ref{loxrem}, $H$ is not virtually cyclic and $H \cap \mathcal{L}_{WPD}(G,S) \neq \emptyset$, so by Proposition \ref{coboundedness}, there exists generating set $W$ of $G$ such that the Cayley graph $\Gamma(G, W)$ is hyperbolic, the action of $G$ on $\Gamma(G, W)$ is acylindrical, and the action of $H$ on $\Gamma(G, W)$ is non-elementary. Then, by Lemma \ref{hlem}, there exist pairwise non-commensurable elements $a_1, \ldots, a_k \in H$ which are loxodromic with respect to the action of $G$ on $\Gamma(G,W)$ such that $E_G(a_i)=\langle a_i \rangle \times K_G(H)$. In particular, we have that $$a_1, \ldots, a_k \in C_G(K_G(H))$$ and, since the elements $a_1, \ldots, a_k$ are pairwise non-commensurable,
\begin{equation}\label{intersect}
E_G(a_i) \cap E_G(a_j) = K_G(H)
\end{equation}
whenever $i \neq j$.

By Lemma \ref{T3}, there exists a generating set $X$ of $G$ such that $$\{E_G(a_1), \ldots, E_G(a_k) \} \hookrightarrow_h (G,X).$$

For each $i \in \{1, \ldots, l \}$ and each $n \leq m_i$, let $Z_i = \{ g_{1,i}, \ldots, g_{m_i,i} \}$ and let
\begin{equation}\nonumber
Y_{i,n}= \{g \in \langle Z_i \rangle \, | \, |g|_{Z_i} \leq 2^{m_i-n}\}.
\end{equation}

I.e., $Y_{i,n}$ is the set of all elements of $G$ which can be expressed as products of $2^{m_i - n}$ or fewer (not necessarily distinct) elements of $\{ g_{1,i}, \ldots, g_{m_i,i} \}$ and their inverses. Let $$Y = \cup_{i=1} ^l Y_{i,1}.$$ Since $Y$ is finite, we may assume by Lemma \ref{finitehypemb} that $X$ contains $Y$.

For each $i \in \{1, \ldots, k\}$, let $\widehat{d}_i$ denote the relative metric in the Cayley graph $$\Gamma(G,X \sqcup (E_G(a_1) \sqcup \ldots \sqcup E_G(a_k))$$ which is induced by the hyperbolic embedding $\{E_G(a_1), \ldots, E_G(a_k) \} \hookrightarrow_h (G,X)$ and corresponds to $E_G(a_i)$ (see Definition \ref{dhat}). Since for each $i \in \{ 1, \ldots, k \}$, the element $a_i$ has infinite order and the metric space $(E_G(a_i), \widehat{d}_i)$ is proper, we may choose $\gamma \in \mathbb{N}$ large enough such that for any $\beta \in \mathbb{Z} \backslash \{0\}$ and any $i \in \{ 1, \ldots, k \}$,  $$\widehat{d}_i((a_i ^{\gamma})^{\beta},1)>2C\max \{m_1, \ldots, m_l \}$$ where $C$ is the constant given by Lemma \ref{ngonC}. By (\ref{keqn}), there exists $j \in \{1, \ldots, l\}$ such that at least $m_j+1$ elements of the set $\{a_1 ^\gamma, \ldots, a_k ^\gamma \}$ satisfy the equation $p_j(x)=1$. Without loss of generality, the elements $a_1 ^\gamma, \ldots, a_{m_j+1} ^\gamma$ satisfy $p_j(x)=1$.

Let $Q$ be the subset of all $q(x) \in G \ast \langle x \rangle$ of the form $$q(x)=h_{1} x^{\beta_{1}} h_{2} x^{\beta_{2}} \ldots h_{N} x^{\beta_{N}}$$ for some $N \leq m_j$ and $\beta_{1}, \ldots, \beta_{N} \in \mathbb{Z}$ such that
\begin{itemize}
\item[(a)] $h_1, h_2, \ldots, h_{N} \in Y_{j,N}$.
\item[(b)] $q(x)=1$ is $C_G(K_G(H))$-equivalent to $p_j(x)=1$ (see Definition \ref{equiv}). 
\end{itemize}
Certainly $Q$ is non-empty since $p_j(x)$ satisfies conditions (a), and (b). Let $q(x)$ be an element of $Q$ whose corresponding value of $N$ is minimal. We claim that $N=1$ for $q(x)$.

Assuming that this is not the case, let $q(x)=h_1 x^{\beta_1} h_2 x^{\beta_2} \ldots h_N x^{\beta_N}$ for some $N \geq 2$ where $h_1, \ldots, h_N \in Y_{j,N}$ and $\beta_1, \ldots, \beta_N \in \mathbb{Z}$. It follows that since $N \geq 2$ is minimal, we have that $h_{1}, \ldots, h_{N} \neq 1$ and $\beta_{1}, \ldots, \beta_{N} \neq 0$, since otherwise, there is a cyclic reduction of $q(x)$ which is strictly shorter, violating the minimality of $N$. For each $i \in \{ 1, \ldots, m_j+1 \}$, we have that $a_i ^{\gamma} \in C_G(K_G(H)) \cap H$ is a solution to the equation $p_j(x)=1$ and hence by (c) to the equation $q(x)=1$. Since $$q(a_i ^{\gamma})=h_1 (a_i ^{\gamma})^{\beta_1} h_2 (a_i ^{\gamma})^{\beta_2} \ldots h_N (a_i ^{\gamma})^{\beta_N}=1,$$ we have by Lemma \ref{tech} that there exists $r_i \in \{1, \ldots, N \}$ for which $h_{r_i} \in E_G(a_i)$. Since this is true for each $i \in \{ 1, \ldots, m_j+1 \}$ and since $N \leq m_j$, there exist $s,t \in \{ 1, \ldots, m_j+1 \}$ with $s \neq t$ such that $r_s=r_t$ so that $h_{r_s} \in E_G(a_s) \cap E_G(a_t) = K_G(H)$ (see (\ref{intersect})).

Without loss of generality, we may assume that $r_s = 2$ (otherwise, take a cyclic permutation of $q(x)$). Then, since $h_2 \in K_G(H)$, the equation $q(x)=1$ (and hence $p_j(x)=1$) is $C_G(K_G(H))$-equivalent to $q'(x)=1$ where $$q'(x)=(h_1 h_2) x^{\beta_1 + \beta_2} \ldots h_N x^{\beta_N}.$$
Furthermore, since $h_1, h_2, h_3, \ldots, h_N \in Y_{j,N}$, we have that $(h_1 h_2), h_3, \ldots, h_N \in Y_{j,N-1}$. Thus we have an element $q'(x) \in Q$ which contradicts our choice of $q(x)$ (namely the minimality of $N$). Therefore, $N=1$ and so $q(x)= h_1x^{\beta_1}$ for some $h_1 \in G$ and $\beta_1 \in \mathbb{Z}$.

It must also be that $\beta_1 = 0$, since otherwise $q(a_1 ^\gamma)=q(a_2 ^\gamma)=1$ implies that there exists a non-zero integer $\beta_1$ such that $a_1 ^{\gamma \beta_1} = a_2 ^{\gamma \beta_1}$, showing that $a_1$ and $a_2$ are commensurable. Therefore $q(x)=h_1 \in G$ and hence it must be that $h_1=1$. The equation $q(x)=1$ is trivially satisfied by all elements of $C_G(K_G(H))$, and therefore each $c \in C_G(K_G(H))$ also satisfies the equation $p_j(x)=1$ by (c). Thus we obtain that $$C_G(K_G(H)) \subseteq V_\lambda = \bigcup_{i = 1} ^l \{ g \in G \, | \, p_i(g)=1 \}$$ as desired.
\end{proof}

We can now prove Theorem \ref{main}:

\begin{proof}[Proof of Theorem \ref{main}]
To show the forward implication, observe that $C_G(K_G(H)) \leq H$ by Theorem \ref{strongmain} and that $H \leq N_G(K_G(H))$ by the definition of $K_G(H)$.

To show the reverse implication, observe that $C_G(K)$ is algebraic by Example \ref{centalg} and that since $C_G(K) \leq H \leq N_G(K)$ and $K$ is finite, it follows that $C_G(K)$ is of finite index in $H$. Then, by Corollary \ref{finindalg}, $H$ is algebraic.
\end{proof}


\section{Applications}\label{appl}


\subsection{Free products}

\begin{proof}[Proof of Corollary \ref{freecor}]
Let $T$ be the Bass-Serre tree corresponding to the graph of groups with vertex groups $A$ and $B$ and trivial edge group. It is easy to see that the natural action of $A \ast B$ on $T$ is acylindrical. (This is a particular case of \cite[Lem. 5.2]{MO}.) By Theorem \ref{main}, either the action of $H$ on $T$ is elementary, or $H$ acts non-elementarily and there exists a finite subgroup of $F \leq A \ast B$ such that $C_{A \ast B}(F) \leq H \leq N_{A \ast B}(F)$.

Before addressing each case, it will be helpful to recall the Kurosh Subgroup Theorem \cite{Kur}, which states that any subgroup $K$ of $A \ast B$ decomposes as a free product
\begin{equation}\label{kurosh}
K = F(X) \ast (\ast_{i \in I} g_i ^{-1} A_i g_i) \ast (\ast_{j \in J} h_j ^{-1} B_j h_j)
\end{equation}
where $X \subseteq G$ freely generates a subgroup of $G$, for all $i \in I$ we have $g_i \in A \ast B$ and $A_i \leq A$, and for all $j \in J$ we have $h_j \in A \ast B$ and $B_j \leq B$.

Assume first that $H$ acts elementarily on $T$. Then by Theorem \ref{actions}, either $H$ is virtually cyclic and contains a loxodromic element or $H$ acts on $T$ with bounded orbits. If $H$ is virtually cyclic, then since $H$ decomposes as a free product as in (\ref{kurosh}), it must be that either $$H \cong \mathbb{Z}_2 \ast \mathbb{Z}_2 \cong D_\infty$$ or $H$ is equal to just a single factor of the free product in (\ref{kurosh}). In the latter case, either $H=F(X)$ (in which case $H$ is cyclic), or $H$ is conjugate to a subgroup of either $A$ or $B$ (which is not possible if $H$ contains a loxodromic element since conjugates of $A$ and $B$ are vertex stabilizers of $T$ and hence act with bounded orbits).

If $H$ acts with bounded orbits, then by \cite[Ch. 2, Cor. 2.8]{BH}, $H$ fixes a point in $T$ and hence a vertex. Since vertex stabilizers of $T$ are conjugates of $A$ and $B$, $H$ is conjugate to a subgroup of either $A$ or $B$.

Now assume that $H$ acts non-elementarily on $T$ and there exists a finite subgroup of $F \leq A \ast B$ such that $C_{A \ast B}(F) \leq H \leq N_{A \ast B}(F)$. If $F$ is trivial, then $H = A \ast B$. If $F$ is non-trivial, then $F$ decomposes as in (\ref{kurosh}), and in particular, $F$ must be equal to some non-trivial $g_i ^{-1} A_i g_i$ or $h_j ^{-1} B_j h_j$. Without loss of generality, assume $F = g_i ^{-1} A_i g_i \neq \{1\}$ for some $A_i \leq A$. We claim that in this case, $H \leq g_i ^{-1} A g_i$.


Let $f \in F \backslash \{1\}$. Since $H \leq N_{A \ast B}(F)$ and $F$ fixes the vertex $g_i ^{-1} A$ of $T$, we have for each $h \in H$ that
$$h^{-1}fh(g_i ^{-1} A)=g_i ^{-1} A$$ and thus $$f(hg_i ^{-1} A)=hg_i ^{-1} A$$ so that $f$ fixes the vertex $hg_i ^{-1} A$ of $T$.

Now observe that each non-trivial element $g \in G \backslash \{1\}$ fixes at most one vertex of $T$. Indeed, if $g \in G$ fixes two vertices, then the unique path between those vertices is also fixed by $g$, so $g$ stabilizes an edge and is therefore trivial. Since $f$ is non-trivial, we must therefore have that $hg_i ^{-1} A = g_i ^{-1} A$ so that $h$ fixes the vertex $g_i ^{-1} A$ and is thus an element of $g_i ^{-1} A g_i$. So $H \leq g_i ^{-1} A g_i$ as desired.
\end{proof}

\subsection{Torsion-free relatively hyperbolic groups}

In this subsection, we prove Corollary \ref{TFRHcor}, but to do so, we will need the two lemmas below:

\begin{lem}\label{algabel}
Let $A$ be an abelian group and let $S$ be an algebraic subset of $A$. Suppose that there exists an infinite-order element $h \in A$ such that $| \langle h \rangle \cap S | = \infty$. Then $S=A$.
\end{lem}

\begin{proof}
It suffices to prove the statement in the case where $S$ is the union of finitely many primitive solution sets, since if $$S = \cap_{i \in I} S_i$$ where $I$ is some indexing set and each $S_i$ is a union of finitely many primitive solution sets, then we will have shown for each $i \in I$ that $S_i=A$, so indeed $S=A$. So let $p_1(x), \ldots, p_n(x)$ be elements of $A \ast \langle x \rangle$ such that $$S = \bigcup_{i = 1} ^n \{ a \in A \, | \, p_i(a)=1 \}.$$ Furthermore, since $A$ is abelian, we may assume that for each $i \in \{1, \ldots, n\}$, the equation $p_i(x)=1$ is of the form $a_i x^{m_i}=1$ for some $a_i \in A$ and $m_i \in \mathbb{Z}$.

Since $| \langle h \rangle \cap S | = \infty$, there exists $i \in \{1, \ldots, n\}$ and $k,l \in \mathbb{N}$ with $k \neq l$ such that $a_i(h^{k})^{m_i}=1$ and $a_i(h^{l})^{m_i}=1$. Then $h^{k m_i}=h^{l m_i}$, and since $h$ has infinite order, it must be that $m_i=0$. But then $a_i=1$ and the equation $p_i(x)=1$ is trivially satisfied by all elements of $A$ so that $S=A$ as desired.
\end{proof}

\begin{cor}
All proper algebraic subgroups of abelian groups are torsion.
\end{cor}

\begin{proof}
If $H\leq A$ is an algebraic subgroup, then for each $h \in H$, either $h$ has finite order or we may apply Lemma \ref{algabel} to $h$ to conclude that $H=A$.
\end{proof}

\begin{lem}\label{abelhypemb}
Suppose that $G$ is a group, $A$ is an abelian hyperbolically embedded subgroup of $G$, $S$ is an algebraic subset of $G$, $H$ is a subgroup of $G$ contained in $S$, and there exists $g \in G$ such that $H \leq g^{-1}Ag$. Then either $H$ is torsion or $g^{-1}Ag \subseteq S$.
\end{lem}

\begin{proof}
Since the Zariski topology is invariant under conjugation by Remark \ref{multhomeom} and since conjugation preserves the orders of group elements, we may assume that $g=1$. As in the proof of Lemma \ref{algabel}, it suffices to prove the lemma in the case when $S$ is the union of finitely many primitive solution sets.

Assume that $H$ is not torsion. The majority of the proof below is similar to the proof of Theorem \ref{strongmain}.

Let $p_1(x), \ldots, p_l(x) \in G \ast \langle x \rangle$ such that $$H \subseteq S = \bigcup_{i = 1} ^l \{ k \in G \, | \, p_i(k)=1 \}.$$ We may assume as in the the proof of Theorem \ref{strongmain} that $p_i(x)$ is of the form $$p_i(x)=g_{1,i} x^{\alpha_{1,i}} g_{2,i} x^{\alpha_{2,i}} \ldots g_{m_i,i} x^{\alpha_{m_i,i}}$$ where $g_{1,i}, \ldots, g_{m_i,i} \in G$ and $\alpha_{1,i}, \ldots, \alpha_{m_i,i} \in \mathbb{Z}$, and if $m_i \geq 2$, then $g_{1,i}, \ldots, g_{m_i,i} \neq 1$ and $\alpha_{1,i}, \ldots, \alpha_{m_i,i} \neq 0$. For each $i \in \{1, \ldots, l \}$ and each $n \leq m_i$, define $Y_{i,n}$ as in the proof of Theorem \ref{strongmain}. Then $A \hookrightarrow_h (G,X)$ for some $X \subseteq G$, and we may assume by Lemma \ref{finitehypemb} that $X$ contains $\cup_{i=1} ^l Y_{i,1}$

Let $\widehat{d}_A$ denote the relative metric in the Cayley graph $\Gamma(G,X \sqcup A)$ induced by the hyperbolic embedding $A \hookrightarrow_h (G,X)$. Since $H$ contains some element $h$ of infinite order and the metric space $(A,\widehat{d}_A)$ is proper, there exists a strictly ascending sequence of natural numbers $\gamma_1, \gamma_2, \gamma_3, \ldots$ such that for any $\beta \in \mathbb{Z} \backslash \{0\}$ and any $i \in \mathbb{N}$, $$\widehat{d}_A((h^{\gamma_i})^{\beta},1)>2C\max \{m_1, \ldots, m_l \}$$ where $C$ is the constant given by Lemma \ref{ngonC}. Since for each $i \in \mathbb{N}$ we have that $h^{\gamma_i} \in H \subseteq S$, there exists $j \in \{1, \ldots, l\}$ and a subsequence of $\{\gamma_i\}_{i \in \mathbb{N}}$ whose corresponding $h^{\gamma_i}$ satisfy the equation $p_j(x)=1$. Without loss of generality, take $\{\gamma_i\}_{i \in \mathbb{N}}$ to be this subsequence.

Let $Q$ be the subset of all $q(x) \in G \ast \langle x \rangle$ of the form $$q(x)=h_{1} x^{\beta_{1}} h_{2} x^{\beta_{2}} \ldots h_{N} x^{\beta_{N}}$$ for some $N \leq m_j$ and $\beta_{1}, \ldots, \beta_{N} \in \mathbb{Z}$ such that
\begin{itemize}
\item[(a)] $h_1, h_2, \ldots, h_{N} \in Y_{j,N}$.
\item[(b)] $q(x)=1$ is $A$-equivalent to $p_j(x)=1$ (see Definition \ref{equiv}). 
\end{itemize}
Certainly $Q$ is non-empty since $p_j(x)$ satisfies conditions (a) and (b). Let $q(x)$ be an element of $Q$ whose corresponding value of $N$ is minimal.

Since $$\widehat{d}_A((h^{\gamma_1})^{\beta},1)>2C\max \{m_1, \ldots, m_l \},$$ we may argue as in the proof of Theorem \ref{strongmain} that if $N \neq 1$ for $q(x)$, then there exists $r \in \{1, \ldots, N\}$ for which $h_r$ is in $A$. Without loss of generality, $r=2$. Letting $$q'(x)=(h_1 h_2) x^{\beta_1 + \beta_2} \ldots h_N x^{\beta_N},$$ we see that since $A$ is abelian, $q'(x)$ is in $Q$, which contradicts the minimality of $N$. Thus it must be that $N=1$ for $q(x)$.

Since $N=1$, we have $h_1(h^{\gamma_1})^{\beta_1}=1$ so that $h_1 \in A$. Then $q(x)=1$ is an equation with coefficients in $A$ which is $A$-equivalent to $p_j(x)=1$, so $\{ k \in G \, | \, p_j(k)=1 \} \cap A$ is an algebraic subset of $A$. Since $h \in A$ has infinite order and $h^{\gamma_i} \in \{ k \in G \, | \, p_j(k)=1 \} \cap A$ for each $i \in \mathbb{N}$, we have that $$|\langle h \rangle \cap \{ k \in G \, | \, p_j(k)=1 \} \cap A| = \infty,$$ so $\{ k \in G \, | \, p_j(k)=1 \} \cap A = A$ by Lemma \ref{algabel}. Thus $A \subseteq S$ as desired.
\end{proof}

We can now prove Corollary \ref{TFRHcor}:

\begin{proof}[Proof of Corollary \ref{TFRHcor}]
Let $\mathcal{H}= \bigsqcup_{\lambda \in \Lambda} H_\lambda$. By \cite[Thm. 5.2]{Osin}, for any finite relative generating set $X$ of $G$, the action of $G$ on $\Gamma(G,X \sqcup \mathcal{H})$ is acylindrical. If $H$ is non-elementary, then $G$ admits a non-elementary acylindrical action on a hyperbolic space and is thus acylindrically hyperbolic, so $H=G$ by Theorem \ref{main}. Otherwise, the action of $H$ on $\Gamma(G,X \sqcup \mathcal{H})$ is elementary. If this is the case, then by Theorem \ref{actions}, $H$ is either elliptic or virtually cyclic (and hence cyclic).

So now let $H$ be elliptic. We want to show that $H \leq g^{-1}H_\lambda g$ for some $g \in G, \lambda \in \Lambda$. Since $G$ is torsion-free, we have by \cite[Thm. 1.14]{ORel} that every element of $G$ which is not conjugate to an element of $H_\lambda$ for some $\lambda \in \Lambda$ acts on $\Gamma(G,X \sqcup \mathcal{H})$ with unbounded orbits. Hence there exists $a \in (H \cap H_\lambda) \backslash \{1\}$ for some $\lambda \in \Lambda$ (unless $H = \{1\}$, in which case we are done). It suffices to show that every $b \in H$ belongs to the same $H_\lambda$.

If there exists $b \in H \backslash H_\lambda$, then by Lemma \ref{finitehypemb} we may assume that $b \in X$ so that $d_{X \sqcup \mathcal{H}}(1,b)=1$. By \cite[Lem. 4.4]{OsinElRel}, there exists a finite subset $F_\lambda \subseteq H_\lambda$ such that if $h \in H_\lambda \backslash F_\lambda$, $b \in G \backslash H_\lambda$, and $d_{X \sqcup \mathcal{H}}(1,b)=1$, then $bh$ acts on $\Gamma(G,X \sqcup \mathcal{H})$ with unbounded orbits. Since $a$ is non-trivial and thus of infinite order, we have that for large enough $n$, $a^n \not\in F_\lambda$ so that $ba^n$ acts on $\Gamma(G,X \sqcup \mathcal{H})$ with unbounded orbits. This contradicts the assumption that $H$ is elliptic. This completes the proof of the first statement of Corollary \ref{TFRHcor}.

To prove the second statement of the corollary, assume that $H$ is conjugate to a subgroup of some abelian $H_\lambda$. Then $H_\lambda \hookrightarrow_h G$, and since $H$ is torsion-free, we have by Lemma \ref{abelhypemb} that $H=\{1\}$ or $H$ is conjugate to $H_\lambda$.
\end{proof}

\subsection{Ascending chains of algebraic subgroups}

In general, ascending chains of algebraic subgroups need not stabilize, even in acylindrically hyperbolic groups. For example, consider the following construction:

\begin{ex}\label{ACCex}
Consider a group with presentation $\langle X | R \rangle$ that contains an infinite ascending chain of subgroups $$K_1 \leq K_2 \leq \ldots$$ where for each $i \in \mathbb{N}$, $K_i \neq K_{i+1}$. Let $$H = \langle X,a, t_1, t_2, \ldots \, | \, R, [X,a]=1, [K_1,t_1]=1, [K_2, t_2]=1, \ldots \rangle.$$ Then $G=H \ast \mathbb{Z}$ is acylindrically hyperbolic with non-elementary acylindrical action on the Bass-Serre tree corresponding to the graph of groups with vertex groups $H$ and $\mathbb{Z}$ and trivial edge group. Furthermore, for each $i \in \mathbb{N}$, $$K_i = \{x \in H \ast \mathbb{Z} \, | \, [x,t_i]=1 \} \cap \{x \in H \ast \mathbb{Z} \, | \, [x,a]=1 \}$$ so that $K_1 \leq K_2 \leq \ldots$ is an ascending chain of algebraic subgroups of $G$ that does not stabilize.
\end{ex}

However, ascending chains of algebraic subgroups in acylindrically hyperbolic groups are subject to the property given by Corollary \ref{ACCcor}, whose proof is given below:

\begin{proof}[Proof of Corollary \ref{ACCcor}]
Assume that condition $(a)$ fails. Then there exists a hyperbolic space $S$ such that $G$ acts on $S$ acylindrically and, by Theorem \ref{actions}, $\cup_{i \in \mathbb{N}} H_i$ contains an element $h$ which is loxodromic with respect to this action. Hence there exists some $I \in \mathbb{N}$ such that $h \in H_I$. Then exactly one of two things is true: either $\langle h \rangle \leq H_i \leq E_G(h)$ for all $i \geq I$, where $\langle h \rangle$ is a finite index subgroup of $E_G(h)$ (in which case the chain clearly stabilizes), or there exists $I' \geq I$ such that $H_i$ acts non-elementarily on $S$ for all $i \geq I'$.

In the latter case, we have by Theorem \ref{main} that for each $i \geq I'$ there exists a finite subgroup $K_i \leq G$ such that $$C_G(K_i) \leq H_i \leq N_G(K_i),$$ and in particular, we saw in the proof of Theorem \ref{main} that $K_i =K_G(H_i)$. So in fact we have that $$C_G(K_G(H_i)) \leq H_i \leq N_G(K_G(H_i)),$$ where $C_G(K_G(H_i))$ is a finite index subgroup of $N_G(K_G(H_i))$ since $K_G(H_i)$ is finite. Moreover, by definition, $K_G(H_{i+1}) \leq K_G(H_i)$ for all $i \geq I'$, and since $K_G(H_{I'})$ is finite, the sequence $\{ K_G(H_i)\}_{i \geq I'}$ stabilizes. Hence there exists $I'' \geq I'$ such that $$C_G(K_G(H_{I''})) \leq H_i \leq N_G(K_G(H_{I''}))$$ for all $i \geq I''$, where $C_G(K_G(H_{I''}))$ is a finite index subgroup of $N_G(K_G(H_{I''}))$, so the chain $H_1 \leq H_2 \leq H_3 \leq \ldots$ must stabilize.
\end{proof}

It should be noted that the group $G$ in Example \ref{ACCex} has an infinite ascending chain of algebraic subgroups because its free factor $H$ also has such a chain. Is it the case then, that for relatively hyperbolic groups, all examples of infinite ascending chains of algebraic subgroups arise in a similar manner? In other words,

\begin{question} Let $\{H_\lambda\}_{\lambda \in \Lambda}$ be a collection of groups with the property that for each \linebreak $\lambda \in \Lambda$, every ascending chain of algebraic subgroups in $H_\lambda$ stabilizes, and let $G$ be a relatively hyperbolic group with peripheral subgroups $\{H_\lambda\}_{\lambda \in \Lambda}$. Is it true that every chain of algebraic subgroups of $G$ stabilizes?
\end{question}

One possible route towards a positive answer is to try to prove that
\begin{itemize}
\item[(a)] If $H$ is a subgroup of $G$ acting eliptically on a relative Cayley graph corresponding to the collection of subgroups $\{H_\lambda\}_{\lambda \in \Lambda}$, then either $A$ is finite or $A$ is conjugate into $H_\lambda$ for some $\lambda \in \Lambda$, and
\item[(b)]
if $A$ is an algebraic subgroup of $G$ such that $A\leq H_\lambda$ for some $\lambda \in \Lambda$, then $A$ is algebraic in $H_\lambda$. 
\end{itemize}
These items, together with Corollary \ref{ACCcor}, would show that every chain stabilizes.

\vspace{1cm}

\noindent \textbf{Bryan Jacobson: } Department of Mathematics, Vanderbilt University, Nashville, TN 37240, U.S.A.\\
E-mail: \emph{bryan.j.jacobson@vanderbilt.edu}

\end{document}